\documentclass[leqno,final]{article}

\usepackage[dvips]{graphicx}
\usepackage{amsmath,amstext,amssymb,hyperref,amsthm}

\numberwithin{equation}{section}
\newcommand{\rem}{{\it Remark }}
\newtheorem{theorem}{Theorem}[section]
\newtheorem{lemma}{Lemma}[section]
\newtheorem{corollary}{Corollary}[section]
\newtheorem{definition}{Definition}[section]

\newcommand{\norm}[1]{\left\Vert#1\right\Vert}
\newcommand{\norml}[2]{\left\Vert#1\right\Vert_{L^2(#2)}}
\newcommand{\norme}[1]{\left\Vert{\hskip -2.7pt}\left\vert #1 \right\vert{\hskip -2.7pt}\right\Vert}

\newcommand{\abs}[1]{\left\vert#1\right\vert}
\newcommand{\pd}[1]{\left\langle #1\right\rangle}
\newcommand{\pdjj}[2]{\left\langle \left[#1\right], \left[#2\right]\right\rangle_e}

\newcommand{\set}[1]{\left\{#1\right\}}

\newcommand{\jm}[1]{\left[#1\right]}
\newcommand{\nfrac}[2]{\displaystyle{\genfrac{}{}{0pt}{}{#1}{#2}}}

\newcommand{\R}{\mathbb{R}}

\newcommand{\db}{\displaybreak[0]}
\newcommand{\nn}{\nonumber}
\newcommand{\ls}{\lesssim}

\newcommand{\be}{\beta}

\newcommand{\De}{\Delta}
\newcommand{\ep}{\varepsilon}
\newcommand{\ga}{\gamma}
\newcommand{\Ga}{\Gamma}
\newcommand{\la}{\lambda}

\newcommand{\na}{\nabla}
\newcommand{\om}{\omega}
\newcommand{\Om}{\Omega}
\newcommand{\pa}{\partial}
\newcommand{\pr}{\prime}

\newcommand{\ta}{\theta}
\newcommand{\vp}{\varphi}

\newcommand{\M}{\mathcal{M}}

\renewcommand{\i}{{\rm\mathbf i}}
\newcommand{\ue}{u_\mathcal{E}}
\newcommand{\ua}{u_\mathcal{A}}

\DeclareMathOperator{\re}{{Re}} 
\DeclareMathOperator{\im}{{Im}}

\newcommand{\p}{\partial}
\newcommand{\T}{\mathcal{T}}
\newcommand{\E}{\mathcal{E}}

\newcommand{\eq}[1]{\begin{align}#1\end{align}}
\newcommand{\eqn}[1]{\begin{align*}#1\end{align*}}

\title{Preasymptotic error analysis of \\higher order FEM and CIP-FEM for\\ Helmholtz equation with high wave number}
\author{
Yu Du
\thanks{Department of Mathematics, Nanjing University, Jiangsu,
210093, P.R. China. ({\tt dynju@qq.com}). }
\and 
Haijun Wu
\thanks{Department of Mathematics, Nanjing University, Jiangsu,
210093, P.R. China. ({\tt hjw@nju.edu.cn}). The work of the second author was
partially supported by the National Magnetic Confinement Fusion Science Program under grant 2011GB105003 and by the NSF of China grants 11071116 and 91130004.}
}

\begin{document}
\date{}
\maketitle


\setcounter{page}{1}

\begin{abstract}
A preasymptotic error analysis of the finite element method (FEM) and some continuous interior penalty finite element method (CIP-FEM) for Helmholtz equation in two and three dimensions is proposed. $H^1$- and $L^2$- error estimates with explicit dependence on the wave number $k$ are derived. In particular, it is shown that if $k^{2p+1}h^{2p}$ is sufficiently small, then the pollution errors of both methods in $H^1$-norm are bounded by $O(k^{2p+1}h^{2p})$, which coincides with the phase error of the FEM obtained by existent dispersion analyses on Cartesian grids, where $h$ is the mesh size, $p$ is the order of the approximation space and is fixed. The CIP-FEM extends the classical one by adding more penalty terms on jumps of higher (up to $p$-th order) normal derivatives in order to reduce efficiently the pollution errors of higher order methods. Numerical tests are provided to verify the theoretical findings and to illustrate great capability of the CIP-FEM in reducing the pollution effect. 
\end{abstract}



\setcounter{page}{1}


\section{Introduction}\label{sec-1} This paper is devoted to preasymptotic error estimates of some continuous interior penalty finite element method (CIP-FEM) and the finite element method (FEM) for the following Helmholtz problem:
\begin{align}
-\De u - k^2 u &=f  \qquad\mbox{in  } \Om,\label{eq1.1a}\\
\frac{\pa u}{\pa n} +\i k u &=g \qquad\mbox{on } \Ga,\label{eq1.1b}
\end{align}
where $\Om\subset \R^d,\, d=2,3$ is a bounded domain with smooth boundary,
$\Ga:=\pa\Om$, $\i=\sqrt{-1}$ denotes the imaginary unit, and $n$
denotes the unit outward normal
to $\pa\Om$. The above Helmholtz problem is an approximation of the
acoustic scattering problem (with time dependence $e^{\i\om t}$) and $k$ is known as the wave number. The Robin boundary condition \eqref{eq1.1b} is known as the
first order approximation of the radiation condition (cf. \cite{em79}).
We remark that the Helmholtz problem \eqref{eq1.1a}--\eqref{eq1.1b} also arises in applications
as a consequence of frequency domain treatment of attenuated scalar waves (cf. \cite{dss94}). 

It is well-known that the finite element method of fixed order for the Helmholtz problem \eqref{eq1.1a}--\eqref{eq1.1b} at high frequencies ($k\gg1$) is subject to the effect of pollution: the ratio of the error of the finite element solution to the error of the best approximation from the finite element space cannot be uniformly bounded with respect to $k$ \cite{Ainsworth04,bs00, bips95, cd03, dbb99,ey11b, fw09, fw11, gm, ghp09, harari97, ib95a,ib97, thompson06, w}. 
More precisely, given that the exact solution $u$ in a space $V$ with norm $\norm{\cdot}_V$ and the  finite element solution $u_h$ in a discrete space $V_h\subset V$, the pollution error may be defined as follows (cf. \cite{ihlenburg98,dgmz12}). Assume that an estimate of the following form holds:
\begin{align}\label{epollutiona}
\frac{\norm{u-u_h}_V}{\norm{u}_V}\le C(k)\inf_{v_h\in V_h}\frac{\norm{u-v_h}_V}{\norm{u}_V}\quad\text{ with }C(k)=C_1+C_2k^\be (kh)^\chi,
\end{align}
where $C_1, C_2, \be>0$, and  $\chi$ are independent of $k$ and the mesh size $h$. Then the finite element solution is said to be \emph{polluted} and the following term is called \emph{pollution error}:
\begin{equation}\label{epollutionb}
C_2k^\be (kh)^\chi\inf_{v_h\in V_h}\frac{\norm{u-v_h}_V}{\norm{u}_V}.
\end{equation}
Clearly, estimating the pollution error is significant both in theory and practice, and it has always been interesting to propose numerical methods which induce less pollution error and consequently, cheap methods \cite{bips95, cd03, dgmz12, fw09, fw11, ghp09,w, zw}. 
We recall that, the term ``asymptotic error estimate" refers to the error estimate without pollution error and the term ``preasymptotic error estimate" refers to the estimate with non-negligible  pollution effect.  

The highly indefinite nature of Helmholtz problem with high wave number makes the error analysis of the FEM very difficult. The standard duality argument (or Schatz argument)  (cf. \cite{ak79, dss94,sch74}) gives only asymptotic error estimates under the mesh condition that $k^2h$ is small enough, but it is too strict for large $k$. In 1990's, Ihlenburg and Babu\v{s}ka \cite{ib95a,ib97} considered the one dimensional problem discretized on equidistant grids, and proved preasymptotic error estimates under the condition that $kh\le C_0$ for some constant less than $\pi$. Based on a profound stability estimate of the exact solution by decomposing it into a nonoscillatory elliptic part and an oscillatory analytic part, and the standard duality argument,  Melenk and Sauter \cite{ms10,ms11} considered one and higher dimensional problems, and showed that the FEM (with fixed $p$) is pollution free under the condition that $k^{p+1}h^p$ is small enough. More recently, Zhu and Wu \cite{zw} gave the first preasymptotic error analysis for higher dimensional problems by combining the stability from \cite{ms10,ms11} and a new modified duality argument. It was shown that the pollution term in the $H^1$ error estimate is $O(k^{2p+1}h^{2p})$, which is exactly of the same order as the phase error obtained by dispersion analysis \cite{Ainsworth04,ib97}, under the mesh condition $k^{p+2}h^{p+1}$ is sufficiently small.  We remark that results on the $hp$ version of the FEM were also obtained in \cite{ib97, ms10,ms11, zw}.

One purpose of this paper is to prove the same preasymptotic error bound for the FEM with fixed order $p$ but under a weaker condition that $k^{2p+1}h^{2p}$ is sufficiently small. Note that this condition is quite practical since a useful numerical solution must has a sufficiently small pollution error which is also $O(k^{2p+1}h^{2p})$. In order to prove this preasymptotic error estimate, we first develope some discrete Sobolev theory on FE spaces. Then we decompose the error of the FE solution $u_h$ as $u-u_h=u-P_hu+P_hu-u_h$ where $P_h$ is an elliptic projection, and we bound $L^2$-norm of $P_hu-u_h$ by its high order discrete Sobolev norms in the duality argument step (instead of its $H^1$ norm as the standard Schatz argument). Note that our new estimates improves the previous results in the case of $p>1$ (cf. \cite{w,zw,ms11}).

 The CIP-FEM, which was first proposed by Douglas and Dupont \cite{dd76} for elliptic and parabolic problems in 1970's and then successfully applied to convection-dominated problems as a stabilization technique \cite{burman05,be07, bh04}, uses the same approximation space as the FEM but modifies the sesquilinear form of the FEM by adding a least squares term penalizing the jump of the normal derivative of the discrete solution at mesh interfaces. Recently the CIP-FEM has shown great potential in solving the Helmholtz problem \eqref{eq1.1a}--\eqref{eq1.1b} with large wave number \cite{w,zw,zbw}. It is absolute stable if the the penalty parameters are chosen as complex numbers with positive imaginary parts, it satisfies an error bound no larger than that of the FEM under the same mesh condition, its penalty parameters may be tuned to greatly reduce the pollution error, and so on. 
 
Another purpose of this paper is to generalize the CIP-FEM by penalizing jumps of higher normal derivatives of the discrete solution at mesh interfaces and to prove the same preasymptotic error estimate as that of the FEM. Note that for the linear case $p=1$, the CIP-FEM remains  unchanged. For higher order case $p>1$, we add more penalty terms on jumps of higher (up to $p$-th order) normal derivatives, because we found by dispersion analysis that the pollution error of the new CIP-FEM for one dimensional problem may be removed completely by choosing appropriate penalty parameters (see Section~\ref{sec-num}), while it is hard to do so for the classical CIP-FEM with only penalty terms on the jump of first order normal derivative.  We use such penalty parameters from one dimensional dispersion analysis to compute a model problem in two dimensions on Cartesian grids and find that the pollution effect is almost invisible for the wave number $k$ up to 1000 for the CIP-FEM with order $p=1, 2, 3$. For simplicity, our theoretical analysis for the CIP-FEM is restrict to the case of \emph{real} penalty parameters. The proofs are quite similar to those of the FEM, except the additional penalty terms should be carefully dealt with. For preasymptotic and asymptotic error analyses of other methods including discontinuous Galerkin methods and spectral methods, we refer to \cite[etc.]{clx13,dgmz12,fw09,fw11,mps13,dpg11,sw07}.   

The remainder of this paper is organized as follows. The CIP-FEM is introduced in Section~\ref{sec-2}. Some preliminary results, including the stability of the continuous solution, the approximation properties of the finite element space, and estimates of the elliptic projection and $L^2$ projection, are cited or proved in Section~\ref{sec-pre}.  In Section~\ref{sec-sob}, we introduce discrete Sobolev norms of arbitrary order by using the discrete elliptic operator and develop useful properties on the discrete Sobolev norms. Section~\ref{sec-fem}  is devoted to the  preasymptotic error analysis of FEM and Section~\ref{sec-cipfem} is devoted to CIP-FEM. In Section~\ref{sec-num}, we simulate a model problem in two dimensions on Cartesian grids by the FEM and CIP-FEM using the ``optimal" penalty parameters for one dimensional problem. The tests verify the theoretical findings and show that the pollution error of the CIP-FEM is much smaller than that of the FEM.

Throughout the paper, $C$ is used to denote a generic positive constant
which is independent of $h$,  $k$, $f$, $g$, and the penalty parameters. We also use the shorthand
notation $A\lesssim B$ and $B\gtrsim A$ for the
inequality $A\leq C B$ and $B\geq CA$. $A\eqsim B$ is a shorthand
notation for the statement $A\lesssim B$ and $B\lesssim A$. We assume that $k\gg 1$ since we are considering high-frequency problems. For the ease of presentation, we assume that $k$ is constant on $\Om$ and that $p=O(1)$ is fixed. We also assume that $\Om$ is a strictly star-shaped domain with an analytic boundary. Here ``strictly star-shaped" means that there exist a point $x_\Om\in\Om$ and a positive constant $c_\Om$ depending only on $\Om$ such that 
\begin{equation}\label{def1}
(x-x_\Om)\cdot n\ge c_\Om,\quad\forall x\in\pa\Om.
\end{equation}
%

\section{Formulations of FEM and CIP-FEM}\label{sec-2}
To formulate the two methods, we first introduce some notation. The standard space, norm and inner product notation
are adopted. Their definitions can be found in \cite{bs08,ciarlet78}.
In particular, $(\cdot,\cdot)_Q$ and $\pd{ \cdot,\cdot}_\Sigma$
for $\Sigma\subset \pa Q$ denote the $L^2$-inner product
on complex-valued $L^2(Q)$ and $L^2(\Sigma)$
spaces, respectively. Denote by $(\cdot,\cdot):=(\cdot,\cdot)_\Om$
and $\pd{ \cdot,\cdot}:=\pd{ \cdot,\cdot}_{\p\Om}$. For simplicity, denote by $\norm{\cdot}_j:=\norm{\cdot}_{H^j(\Om)}$ and $\abs{\cdot}_j:=\abs{\cdot}_{H^j(\Om)}$.

Let $\T_h$ be a curvilinear triangulation of $\Om$ (cf. \cite{ms10, ms11, Monk03}). For any $K\in \T_h$, we define $h_K:=\mbox{diam}(K)$. 
Similarly, for each edge/face $e$ of $K\in \T_h$, define $h_e:=\mbox{diam}(e)$. Let $h=\max_{K\in\T_h}h_K$. Assume that $h_K\eqsim h$. Denote by 
$\widehat K$ the reference element and by $F_K$ the element maps from $\widehat K$ to $K\in \T_h$.
Let $V_h$ be the approximation space of continuous piecewise mapped $p$-th order polynomials, that is,
$$V_h:=\set{v_h \in H^1(\Omega):\ v_h\mid_K\circ F_K \in \mathcal P_p(\widehat K), \forall K \in \T_h},$$
where $\mathcal P_p(\widehat K)$ denotes the set of all polynomials whose degrees do not exceed $p$ on $\widehat K$. 

We remark that the theoretical results of this paper also hold for finite element discretizations on curvilinear Cartesian meshes or isoparametric finite element approximations \cite{bs08}. 

\subsection{FEM} Introduce the following sesquilinear form
\eq{a(u,v)=(\na u,\na v), \quad\forall u,v\in H^1(\Om)\label{ea}.}
The variational problem to \eqref{eq1.1a}--\eqref{eq1.1b} reads as: Find $u\in H^1(\Om)$ such that
\eq{a(u,v) - k^2(u,v) +\i k \pd{ u,v}
=(f,v)+\pd{g, v},\qquad\forall v\in H^1(\Om).\label{evp}}
The FEM is defined by: Find $u_h\in V_h$ such that
\eq{a(u_h,v_h) - k^2(u_h,v_h) +\i k \pd{ u_h,v_h}
=(f,v_h)+\pd{g, v_h},\qquad\forall v_h\in V_h.\label{eFEM}}

The following norm on $H^1(\Om)$ is useful for the subsequent analysis:
\eq{\norme{v}:=&\big(\norm{\na v}_0^2+k^2\norm{v}_0^2\big)^{\frac12}\label{enorme}.}
Noting from the trace inequality that 
\eq{k\norml{v}{\Ga}^2\ls k\norm{v}_0\norm{v}_1\ls k^2\norm{v}_0+\norm{v}_1^2\ls \norme{v}^2,\label{etr}}
we have the following continuity estimate for the sesquilinear form of the FEM:
\eq{\abs{a(u,v) - k^2(u,v) +\i k \pd{ u,v}}&\ls \norme{u}\norme{v}, \quad\forall u,v\in H^1(\Om).\label{eac}}
\subsection{CIP-FEM} Let $\E_h^I$ be the set of all interior edges/faces of $\T_h$. For every $e=\p K\cap \p K^\pr\in\E_h^I$, let $n_e$ be a unit normal vector to $e$ and  define the jump $\jm{v}$ of $v$ on $e$ as $\jm{v}|_{e}:= v|_{K^\pr}-v|_{K}.$

We define the ``energy" space $V$ and the sesquilinear
form $a_\ga(\cdot,\cdot)$ on $V\times V$ as follows:
\begin{align}
V&:=H^1(\Om)\cap\prod_{K\in\T_h} H^{p+1}(K), \nonumber \\
\label{eah}
a_\ga(u,v)&:=a(u,v)+  J(u,v)\qquad\forall\, u, v\in V,\\
J(u,v)&:=\sum_{j=1}^p\sum_{e\in\E_h^{I}}\ga_{j,e}\, h_e^{2j-1} \pdjj{\frac{\pa^j u}{\pa n_e^j}}{\frac{\pa^j v}{\pa n_e^j}},
\label{eJ}\end{align}
where  $\ga_{j,e}, e\in\E_h^I$ are 
numbers with nonnegative imaginary parts to be specified latter. 
It is clear that $J(u,v)=0$ if $u\in H^{p+1}(\Om)$ and $v\in V$. Therefore, if $u\in H^{p+1}(\Om)$ is the solution of \eqref{eq1.1a}--\eqref{eq1.1b}, then 
\begin{equation*}
a_\ga(u,v) - k^2(u,v) +\i k \pd{ u,v}
=(f,v)+\pd{g, v},\qquad\forall v\in V.
\end{equation*}
Then the CIP-FEM is defined as follows: Find $u_h\in V_h$ such that
\begin{equation}\label{ecipfem}
a_\ga(u_h,v_h) - k^2(u_h,v_h) +\i k \pd{ u_h,v_h}
=(f, v_h)+\pd{g, v_h},  \qquad\forall v_h\in V_h.
\end{equation}

\rem 2.1.\ 
(a) The terms in $J(u,v)$ are so-called penalty terms.
The penalty parameters in $J(u,v)$ are $\ga_{j,e}$.  Clearly, if the parameters $\ga_{j,e}\equiv 0$, then the CIP-FEM becomes the standard FEM.

(b) Penalizing the jumps of normal derivatives 
 was used early by Douglas and Dupont \cite{dd76} for second order PDEs and by Babu{\v{s}}ka and Zl\'amal \cite{bz73} for fourth order PDEs in
the context of $C^0$ finite element methods, by Baker \cite{baker77} for fourth order PDEs
and by Arnold \cite{arnold82} for second order
parabolic PDEs in the context of IPDG methods. 

(c) Our CIP-FEM \eqref{ecipfem} extends the classical CIP-FEM \cite{dd76, burman05, w,zw} which penalizing only the jumps of the first normal derivatives. We consider such extension for scattering problems because  more penalty terms are helpful for reducing the pollution effects of higher order methods (see Section~\ref{sec-num}).

(d) The classical CIP-FEM was analyzed by Wu and Zhu in \cite{w,zw} for the Helmholtz problem \eqref{eq1.1a}--\eqref{eq1.1b} and proved to be absolute stable for penalty parameters with positive imaginary parts. Optimal order preasymptotic error estimates were also derived under the mesh condition that $k^{p+2}h^{p+1}$ is small enough. In this paper we will prove that the optimal order preasymptotic error estimates still hold, for the new CIP-FEM including the classical one, when $k^{2p+1}h^{2p}$ is sufficiently small.

(e) In this paper we consider the scattering problem with time dependence $e^{\i\om t}$, that is, the sign before $\i$ in \eqref{eq1.1b} is positive. If we consider the scattering problem with time dependence $e^{-\i\om t}$, that is, the sign before $\i$ in   \eqref{eq1.1b} is negative, then the penalty parameters should be complex numbers with  nonpositive imaginary parts.

We also need the following norms on the space $V$:
\begin{align}\label{e2.5}
\abs{v}_{1,\ga}:=&\bigg( \norm{\na v}_0^2+\sum_{j=1}^p\sum_{e\in\E_h^{I}}\abs{\ga_{j,e}}\, h_e^{2j-1}\norm{\left[\frac{\pa^j v}{\pa n_e^j}\right]}_{L^2(e)}^2\bigg)^{1/2},\\
\norm{v}_{1,\ga}:=&\big(\abs{v}_{1,\ga}^2+\norm{v}_0^2\big)^{1/2},\label{e2.5b}\\
 \norme{v}_\ga:=&\big(\abs{v}_{1,\ga}^2+k^2\norm{v}_0^2\big)^{\frac12}.\label{e2.5c}
\end{align}
Noting that the exact solution may not be in $V$, we introduce the following functions to measure the errors of discrete approximations. 
\begin{align}\label{e2.6a}
E_\ga(v,v_h):=&\bigg( \norm{v-v_h}_1^2+\sum_{j=1}^p\sum_{e\in\E_h^{I}}\abs{\ga_{j,e}}\, h_e^{2j-1}\norm{\left[\frac{\pa^j v_h}{\pa n_e^j}\right]}_{L^2(e)}^2\bigg)^{1/2},\\
\mathbb{E}_\ga(v,v_h):=&\big(E_\ga(v,v_h)^2+k^2\norm{v-v_h}_0^2\big)^{\frac12},\quad \forall v\in H^1(\Om), v_h\in V_h.\label{e2.6b}
\end{align}
Clearly, $E_\ga(v,v_h)=\norm{v-v_h}_{1,\ga}$ and $\mathbb{E}_\ga(v,v_h)=\norme{v-v_h}_\ga$ if $v\in H^{p+1}(\Om), v_h\in V_h$.

In the next sections, we shall consider the preasymptotic stability and
error analysis for the above FEM and the CIP-FEM. 

\section{Preliminary lemmas}\label{sec-pre} In this section, we first recall stability estimates of the continuous problem. Then we introduce approximation estimates of the discrete space $V_h$, in particular, the error estimates of the elliptic projection and $L^2$ projection in negative norms.  
\subsection{Stability estimates of the continuous problem}
The following lemma (cf. \cite[Theorem 4.10]{ms11}) says that the solution $u$ to the continuous problem \eqref{eq1.1a}--\eqref{eq1.1b} can be decomposed into the sum of an elliptic part and an analytic part $u=\ue+\ua$ where $\ue$ is usually non-smooth but the $H^2$-bound of $\ue$ is independent of $k$ and $\ua$ is oscillatory but the $H^j$-bound of $\ua$ is available for any integer $j\ge 0$.  
\begin{lemma}\label{depcomposition} Assume that $\Om$ is a strictly star-shaped domain with an analytic boundary. Suppose $f\in L^2(\Om)$ and $g\in H^{1/2}(\Ga)$.
Then the solution $u$ to the problem \eqref{eq1.1a}--\eqref{eq1.1b}
can be written as $u=\ue+\ua$, and satisfies
\begin{align}
 |\ue|_j&\ls k^{j-2}C_{f,g},\quad j=0,1,2,\label{u_E}\\
 \abs{u_{\mathcal{A}}}_j&\ls k^{j-1}C_{f,g},\quad\forall j\in \mathbb{N}_0. \label{u_A}
\end{align}
Here $C_{f,g}:=\|f\|_0+\|g\|_{H^{1/2}(\Ga)}$.
\end{lemma}

\rem 3.1.\ 
It was shown earlier that (see \cite{cf06,hetmaniuk07, melenk95b}) 
$$k^2\norm{u}_0+k\norm{u}_1+\norm{u}_2\lesssim k \big(\|f\|_0+\|g\|_{L^2(\Ga)}\big)+\|g\|_{H^{1/2}(\Ga)}.$$ 

\begin{lemma}\label{lstau} Assume that $\Om$ is a strictly star-shaped domain with an analytic boundary. Suppose $s\ge 2$ and $f\in H^{s-2}(\Om)$ and $g\in H^{s-3/2}(\Ga)$.
Then the solution $u$ to the problem \eqref{eq1.1a}--\eqref{eq1.1b}
 satisfies the following stability estimate.
\begin{align}
 \norm{u}_s&\ls k^{s-1}C_{s-2,f,g},\label{estau}
\end{align}
where $C_{s-2,f,g}:=\|f\|_0+\|g\|_{L^2(\Ga)}+\sum_{j=0}^{s-2}k^{-(j+1)}\big(\|f\|_j+\|g\|_{H^{j+1/2}(\Ga)}\big)$. 
\end{lemma}
\begin{proof} We prove this lemma by induction. From Remark 3.1, \eqref{estau} holds for $s=2$. Next we suppose that
 \begin{align}
 \norm{u}_l&\ls k^{l-1}C_{l-2,f,g}, \quad 2\le l\le s-1.\label{estauj}
\end{align}
Note that the continuous problem \eqref{eq1.1a}--\eqref{eq1.1b} can be rewritten as
\begin{align*}
-\De u + u &=(k^2+1)u+f  \quad\mbox{in  } \Om,\qquad\frac{\pa u}{\pa n} = -\i k u +g \quad\mbox{on } \Ga.
\end{align*}
The standard regularity estimate for Poisson equation with Neumann boundary condition \cite{gt01} and the trace inequality imply that
\begin{align*}
\norm{u}_s&\ls \norm{(k^2+1)u+f}_{s-2}+\norm{-\i k u +g}_{H^{s-3/2}(\Ga)} \\
&\ls k^2\norm{u}_{s-2} + \norm{f}_{s-2} + k\norm{u}_{s-1} +\norm{g}_{H^{s-3/2}(\Ga)} \\
&\ls k^{s-1}(\|f\|_0+\|g\|_{L^2(\Ga)}\big) + \sum_{j=0}^{s-2}k^{s-j-2}\big(\|f\|_j+\|g\|_{H^{j+1/2}(\Ga)}\big) \\
&\ls k^{s-1}C_{s-2,f,g}.
\end{align*}
Then the proof is completed by induction.
\end{proof}

\subsection{Approximation properties} In this subsection we consider to approximate the solution $u$ to the problem \eqref{eq1.1a}--\eqref{eq1.1b} by finite element functions in $V_h$. 

The following result is well-known.
\begin{lemma}\label{lapprox1}
Let $1\le s\le p+1$. Suppose $u \in H^s(\Omega)$. Then there exists $\hat{u}_h \in V_h$ such that
\begin{align}\label{elapprox1a}
\norm{u-\hat{u}_h}_0+h\norm{u-\hat{u}_h}_1&\lesssim  h^s|u|_s.
\end{align}
\end{lemma}
\begin{proof} $\hat{u}_h$ may be chosen as the standard Lagrange interplant if $s\ge 2$ and as the Scott-Zhang interpolant otherwise  (cf. \cite{bs08}).
\end{proof}

If $u$ is the exact solution satisfying the decomposition $u=\ue+\ua$ as in Lemma~\ref{depcomposition}, then we may approximate $u$ by $\hat u_h=\widehat\ue_h+\widehat\ua_h$ to show the following estimate (cf. \cite{ms10,ms11}).
 \begin{lemma}\label{error2}
Let $u$ be the solution to the problem \eqref{eq1.1a}-\eqref{eq1.1b}. Suppose $f\in L^2(\Om)$ and $g\in H^{1/2}(\Ga)$. Then
there exists $\hat{u}_h \in V_h$ such that
\begin{align}
\|u-\hat{u}_h\|_0+h\norme{u-\hat{u}_h}&\lesssim \big(h^2+h(kh)^p\big) C_{f,g}, \label{u_err0}
\end{align}
where $C_{f,g}$ are defined in Lemmas~\ref{depcomposition}.
\end{lemma}

Define the elliptic projection $P_h$ as follows.
\eq{a(P_h\vp, v_h)+(P_h\vp,v_h)=a(\vp, v_h)+(\vp,v_h), \quad\forall v_h\in V_h,\label{ePh}}
where $a$ is defined in \eqref{ea}.
Then we have the following error estimates in $H^1$, $L^2$, and negative norms \cite{bs08}.
\begin{lemma}\label{lPh} For any $-1\le j\le p-1$ and $\vp\in H^1(\Om)$,
\eqn{\norm{\vp-P_h\vp}_{-j}\ls h^{j+1}\inf_{\vp_h\in V_h}\norm{\vp-\vp_h}_1.}
\end{lemma}

Similarly, for the $L^2$ projection $Q_h$ defined by 
\[(Q_h\vp,v_h)=(\vp,v_h),\quad\forall v_h\in V_h,\]
we have the following lemma.
\begin{lemma}\label{lQh} For any $0\le j\le p+1$,
\eqn{\norm{\vp-Q_h\vp}_{-j}\ls h^{j}\inf_{\vp_h\in V_h}\norm{\vp-\vp_h}_0.}
\end{lemma}
\begin{proof} For any $v\in H^j(\Om)$, from Lemma~\ref{lapprox1}, there exists $\hat v_h\in V_h$ such that
\[\norm{v-\hat v_h}_0\ls h^j\norm{v}_j.\]
Then
\eqn{(\vp-Q_h\vp,v)=(\vp-Q_h\vp,v-\hat v_h)\ls\norm{\vp-Q_h\vp}_0 h^j\norm{v}_j\ls h^j\inf_{\vp_h\in V_h}\norm{\vp-\vp_h}_0\norm{v}_j,}
which completes the proof of the lemma.
\end{proof}

\section{Discrete elliptic operator and discrete Sobolev norms}\label{sec-sob} Noting that a discrete function in $V_h$ is usually not in $H^2$, its high order Sobolev norms may not exist. In this section we introduce discrete Sobolev norms of arbitrary order and discuss relationships between the discrete and standard Sobolev norms. 

Define $A_h: V_h\mapsto V_h$ by
\eq{(A_hv_h,w_h)=a(v_h,w_h)+(v_h,w_h),\quad\forall v_h,w_h\in V_h.\label{eAh}}
Note that $A_h$ is a discrete version of the elliptic operator $A:=-\De +I$ from $D(A):=\set{v\in H^2(\Om):\, \frac{\pa v}{\pa n}=0\text{ on }\pa\Om}$ to $L^2(\Om)$. Clearly,
\eq{(Av,w)=a(v,w)+(v,w),\quad\forall v\in D(A),w\in H^1(\Om).\label{eA}}
Denote the eigenvalues of the operators $A$ and $A_h$ by 
\[\la_1<\la_2<\cdots, \quad\text{and}\quad \la_{1h}<\la_{2h}<\cdots<\la_{\dim(V_h)h}, \quad\text{respectively.}\]
Clearly, the eigenvalues are positive and the corresponding eigenfunctions  denoted by
\[\phi_1,\phi_2,\cdots,\quad\text{and}\quad \phi_{1h},\phi_{2h},\cdots,\phi_{\dim(V_h)h},\]
 form 
 orthogonal bases of the spaces $L^2(\Om)$ and $V_h$, respectively.
 For any real number $j$ we define $A^j$ and $A_h^{j}$ as follows.
\eq{\text{For }v=\sum_{m=1}^\infty a_m\phi_m, &\text{ let } A^jv=\sum_{m=1}^\infty\la_m^ja_m\phi_m;\label{eAj}\\
\text{For }v_h=\sum_{m=1}^{\dim(V_h)}a_m\phi_{mh}, &\text{ let } A_h^jv_h=\sum_{m=1}^{\dim(V_h)}\la_{mh}^ja_m\phi_{mh}.\label{eAhj}} 

 Define the following norm on $D(A^{j/2})$, the domain of the operator $A^{j/2}$:
\eq{\norm{v}_{j*}:=\norm{A^{j/2}v}_0,\quad\forall v\in D(A^{j/2}).\label{enormj}}
Then the definition of $A$ and
the  shift estimates for elliptic differential equations \cite{gt01} show that for any (fixed) integer $j\ge0$,
\eq{\norm{v}_{j}\eqsim \norm{v}_{j*},\quad\forall v\in D(A^{j/2}).\label{enormAj}}
Clearly, $D(A^{j/2}) \subseteq H^j(\Om) \text{ if } j\ge 0$. Note that, for $j=-m<0$, 
\eqn{\norm{v}_j&=\sup_{w\in H^m(\Om)}\frac{(v,w)}{\norm{w}_m}\ge \sup_{w\in D(A^{m/2})}\frac{(v,w)}{\norm{w}_m}\gtrsim\sup_{w\in  D(A^{m/2})}\frac{(A^{-m/2}v,A^{m/2}w)}{\norm{A^{m/2}w}_0}=\norm{v}_{j*}.} 
We have, for any integer $j<0$,    
\eq{ \norm{v}_{j*}\ls\norm{v}_{j},\quad\forall v\in H^j(\Om).\label{enormAj-}}

Introduce the following discrete $H^j$ norms on $V_h$ for any integer $j$:
\eq{\norm{v_h}_{j,h}:=\norm{A_h^{j/2}v_h}_0.\label{enormh}}
It is clear that
\eq{\norm{v_h}_{0,h}=\norm{v_h}_0=\norm{v_h}_{0*},\quad\norm{v_h}_{1,h}=\norm{v_h}_1=\norm{v_h}_{1*}, \quad\forall v_h\in V_h.\label{enormh2}}

The following lemma gives some inverse estimates for discrete functions.
\begin{lemma}\label{lAhinv}For any integer $j$,
\eqn{\norm{v_h}_{j,h}\ls h^{-1}\norm{v_h}_{j-1,h},\quad\forall v_h\in V_h.}
\end{lemma}
\begin{proof}
From the definition of the discrete norm $\norm{\cdot}_{j,h}$ (see \eqref{enormh}), it suffices to prove the inverse estimate in discrete norm for $j=1$ which follows from \eqref{enormh2} and the inverse inequality in standard Sobolev norm:
\[\norm{v_h}_{1,h}=\norm{v_h}_1\ls h^{-1}\norm{v_h}_0=h^{-1}\norm{v_h}_{0,h}.\]
The proof is completed.
\end{proof}


The following lemma gives a relationship between the non-positive  discrete norms and standard  norms  of discrete functions.
\begin{lemma}\label{lAhnorm} For any integer $0\le j\le p+1$, we have 
\eq{\norm{v_h}_{-j,h}&\ls \sum_{m=0}^{j}h^{j-m}\norm{v_h}_{-m}, \quad\forall v_h\in V_h.\label{enormh-ja}
  }
\end{lemma} 
\begin{proof} 
From  \eqref{enormAj-} and \eqref{enormh2}, it suffices to show that 
\eq{\norm{v_h}_{-j,h}&\ls \sum_{m=0}^{j}h^{j-m}\norm{v_h}_{-m*}, \quad\forall v_h\in V_h.\label{enormh-j}}
 Let $z_h=A_h^{-1}v_h$ and $z=A^{-1}v_h$. We have
\eq{a(z_h,w_h)+(z_h,w_h)&=(v_h,w_h),\quad w_h\in V_h,\label{ezh}\\
a(z,w)+(z,w)&=(v_h,w),\quad w\in H^1(\Om). \label{ez}}
From Lemma~\ref{lPh}, Lemma~\ref{lapprox1}, and \eqref{enormAj-}, we have for any $-1\le m\le p-1$,
\eqn{
\norm{z-z_h}_{-m*}&\ls\norm{z-z_h}_{-m}\ls h^{m+1}\inf_{\vp_h\in V_h}\norm{z-\vp_h}_1\ls h^{m+2}\norm{z}_2\ls h^{m+2}\norm{v_h}_0.}
Therefore, for $-1\le m\le p-1$,
\eq{\norm{A_h^{-1}v_h}_{-m*}\ls \norm{A^{-1}v_h}_{-m*}+h^{m+2}\norm{v_h}_0
\ls \norm{v_h}_{-(m+2)*}+h^{m+2}\norm{v_h}_0.\label{elnormh-ja}}

If $j=2l\le p+1$ is even, then by recursive use of \eqref{elnormh-ja} we have
\eqn{\norm{v_h}_{-j,h}=&\norm{A_h^{-l}v_h}_0=\norm{A_h^{-1}(A_h^{-l+1}v_h)}_0\\
\ls &\norm{A_h^{-l+1}v_h}_{-2*}+h^2\norm{A_h^{-l+1}v_h}_0\\
\ls& \norm{A_h^{-l+2}v_h}_{-4*} +h^2 \norm{A_h^{-l+2}v_h}_{-2*}+h^4\norm{A_h^{-l+2}v_h}_0\\
\ls&\cdots\cdots\\
\ls& \norm{v_h}_{-j*}+h^2\norm{v_h}_{-(j-2)*}+\cdots+h^j\norm{v_h}_0.}
That is, \eqref{enormh-j} holds for $j=2l$. 

Next we consider the case that $j=2l+1\le p+1$ is even. Noting that $\norm{z_h}_1=\norm{P_hz}_1\le\norm{z}_1$, from \eqref{enormh2}, we conclude that
\eqn{\norm{A_h^{-1/2}v_h}_0^2=(v_h,A_h^{-1}v_h)=(A_hz_h,z_h)=\norm{z_h}_{1}^2\le\norm{z}_{1}^2=\norm{z}_{1*}^2 = \norm{v_h}_{-1*}^2.}
Therefore,
\eqn{\norm{v_h}_{-j,h}=&\norm{A_h^{-l-1/2}v_h}_0=\norm{A_h^{-1/2}(A_h^{-l}v_h)}_0\ls\norm{A_h^{-l}v_h}_{-1*}.}
Again a recursive use of \eqref{elnormh-ja} implies that \eqref{enormh-j} holds.
%
This completes the proof of the lemma. 
\end{proof}

\rem 4.1. 
 {\rm (a)} It would be of independent interest to investigate further properties of discrete Sobolev norms defined as above, such as, embedding inequalities, trace inequalities, and so on. Here we list merely the useful properties  for the analysis of the paper.  

{\rm (b)} There have been some other theories of discrete Sobolev spaces in the literature, for example, the theories applied to finite difference methods \cite{bus83,chou91,ls95} and the theories applied to discontinuous Galerkin methods \cite{be08,de10,fln_arx}.

\section{Preasymptotic error analysis of FEM}\label{sec-fem} 
One crucial step in asymptotic error analyses  of FEM for scattering problems is performing the 
 duality argument (or Aubin-Nitsche trick) (cf. \cite{ak79, dss94,ib97,ms10,ms11,sch74}). This argument is usually used to estimate the $L^2$-error of the finite element solution by its $H^1$-error. Based on the standard duality argument, the stability estimate in Remark~2.1 leads to asymptotic error estimate only under the condition that $k^2h$ is small enough, while the stability of Melenk and Sauter \cite{ms10,ms11} (cf. Lemma~\ref{depcomposition}) leads to pollution-free estimates under the condition that $k^{p+1}h^p$ is sufficiently small instead.  In \cite{zw}, Zu and Wu develop a modified duality argument which uses some special designed elliptic projections in the duality-argument step so that we can bound the $L^2$-error of the discrete solution by using the errors of the elliptic projections of the exact solution $u$ and obtain the first preasymptotic error estimates for the FEM in higher dimensions under the condition that $k^{p+2}h^{p+1}$ is sufficiently small. In this section, we modify the duality argument further by decomposing the error $u-u_h$ into a sum of $\rho:=u-P_hu$ and $\ta_h:=P_hu-u_h$ and bounding $L^2$-norm of $\ta_h$ by its high order discrete Sobolev norms in the duality argument step, so that we can derive optimal order preasymptotic error estimates under the condition that $k^{2p+1}h^{2p}$ is sufficiently small. This improves the previous results in the case of $p>1$.
 
\begin{theorem}\label{thm-err-1}
Let $u$ and $u_h$ denote the solutions to \eqref{eq1.1a}--\eqref{eq1.1b} and \eqref{eFEM}, respectively. Then there exists a constant $C_0$ independent of $k$ and $h$, such that if 
\begin{equation}\label{econd1}
k(kh)^{2p}\le C_0, 
\end{equation}
then the following error estimates hold:
\begin{align}\label{error-eh1}
\norme{u-u_h}&\lesssim \big(1+k(kh)^p\big)\inf_{z_h \in V_h}\norme{u-z_h},\\
\norm{u-u_h}_0&\lesssim \big(h+(kh)^p\big)\inf_{z_h \in V_h}\norme{u-z_h}.\label{error-eh2}
\end{align}
\end{theorem}
\begin{proof} Suppose $kh\ls 1$. Let $P_hu$ be the elliptic projection of $u$ defined as \eqref{ePh} and let $$e_h:=u-u_h=(u-P_hu)+(P_hu-u_h):=\rho+\ta_h.$$ From Lemma~\ref{lPh}, $\rho$ may be bounded as follows:
\eq{\norm{\rho}_{-j}\ls h^{j+1}\norme{\rho}\ls h^{j+1}\inf_{z_h \in V_h}\norme{u-z_h},\quad 0\le j\le p-1.\label{erho}}
It remains to estimate $\ta_h$.
From \eqref{evp} and \eqref{eFEM} we have the following Galerkin orthogonality,
\begin{equation*}
a(e_h,v_h)-k^2(e_h,v_h)+\i k\langle e_h,v_h \rangle=0,\quad\forall v_h\in V_h.
\end{equation*}
Therefore from \eqref{ePh},
\begin{equation}\label{ecipfemorth}
a(\ta_h,v_h)-k^2(\ta_h,v_h)+\i k\pd{\ta_h,v_h}=(k^2+1)(\rho,v_h)-\i k\pd{\rho,v_h},\quad\forall v_h\in V_h.
\end{equation}

\emph{Step 1.} In this step, we bound $\norml{\ta_h}{\Ga}$ by the $(p-1)$-th order discrete norm of $\ta_h$. Let $v_h=\ta_h$ in \eqref{ecipfemorth} and take the imaginary part of the result equation to obtain
\eqn{k\norml{\ta_h}{\Ga}^2&=\im \big((k^2+1)(Q_h\rho,\ta_h)\big)-\re\big(k\pd{\rho,\ta_h}\big)\\
&\le(k^2+1)\norm{Q_h\rho}_{1-p,h}\norm{\ta_h}_{p-1,h}+k\norml{\rho}{\Ga}\norml{\ta_h}{\Ga}.}
From Lemmas~\ref{lPh} and \ref{lQh} with $\vp_h=P_hu$, we have
\eqn{\norm{Q_h\rho}_{1-p,h}=\norm{Q_hu-u+u-P_hu}_{1-p,h}\ls h^p\norme{\rho}.}
On the other hand, noting that $\norm{\rho}_{L^2(\Ga)}\ls\norm{\rho}_0^{1/2}\norm{\rho}_1^{1/2}\ls h^{1/2}\norme{\rho}$, it follows from the Young's inequality that
\eqn{k\norml{\rho}{\Ga}\norml{\ta_h}{\Ga}&\le\frac{k}{2}\norml{\rho}{\Ga}^2+\frac{k}{2}\norml{\ta_h}{\Ga}^2
\le Ckh\norme{\rho}^2+\frac{k}{2}\norml{\ta_h}{\Ga}^2.}
By combining the above three estimates we obtain
\eq{\label{etaGa}\norml{\ta_h}{\Ga}^2\ls kh^p\norme{\rho}\norm{\ta_h}_{p-1,h}+h\norme{\rho}^2\ls k^2h^{2p-1}\norm{\ta_h}_{p-1,h}^2+h\norme{\rho}^2.}

\emph{Step 2.} In this step, we bound the high order discrete norms of $\ta_h$ by its $L^2$-norm.
 From the definition of $A_h$, \eqref{ecipfemorth} can be rewritten as:
\eqn{(A_h\ta_h, v_h)=(k^2+1)(\ta_h,v_h)+(k^2+1)(Q_h\rho,v_h)-\i k\pd{\ta_h,v_h}-\i k\pd{\rho,v_h},\;\forall v_h\in V_h.}
Given any integer $1\le m\le p$, take $v_h=A_h^{m-1}\ta_h$ in the above equation to obtain:
\eqn{\norm{\ta_h}_{m,h}^2=&(k^2+1)\norm{\ta_h}_{m-1,h}^2+(k^2+1)(A_h^{(m-1)/2}Q_h\rho,A_h^{(m-1)/2}\ta_h)\\
&-\i k\pd{\ta_h,A_h^{m-1}\ta_h}-\i k\pd{\rho,A_h^{m-1}\ta_h}.}
 Moreover, from the trace and inverse inequalities (see Lemma~\ref{lAhinv}) and \eqref{etaGa}, we have, 
\eqn{\abs{\pd{\ta_h,A_h^{m-1}\ta_h}}&\ls \norml{\ta_h}{\Ga}\norml{A_h^{m-1}\ta_h}{\Ga}\ls\norml{\ta_h}{\Ga}h^{-1/2}\norm{\ta_h}_{2m-2,h}\\
&\ls\norml{\ta_h}{\Ga}h^{-m+1/2}\norm{\ta_h}_{m-1,h}\\
&\ls\big(kh^{p-m}\norm{\ta_h}_{p-1,h}+h^{1-m}\norme{\rho}\big)\norm{\ta_h}_{m-1,h}\\
&\ls\big(k\norm{\ta_h}_{m-1,h}+h^{1-m}\norme{\rho}\big)\norm{\ta_h}_{m-1,h},\db\\
\abs{\pd{\rho,A_h^{m-1}\ta_h}}&\ls \norml{\rho}{\Ga}\norml{A_h^{m-1}\ta_h}{\Ga}\ls h^{1/2}\norme{\rho}h^{-m+1/2}\norm{\ta_h}_{m-1,h} \\
&\ls h^{1-m}\norme{\rho}\norm{\ta_h}_{m-1,h},}
Therefore for $1\le m\le p$,
\eqn{\norm{\ta_h}_{m,h}^2\ls& k^2\norm{\ta_h}_{m-1,h}^2+k^2\norm{Q_h\rho}_{m-1,h}\norm{\ta_h}_{m-1,h}\\
&+\big(k\norm{\ta_h}_{m-1,h}+h^{1-m}\norme{\rho}\big)k\norm{\ta_h}_{m-1,h}.}
which implies by the Young's inequality that
\eqn{\norm{\ta_h}_{m,h}\ls k\norm{\ta_h}_{m-1,h}+k\norm{Q_h\rho}_{m-1,h}+h^{1-m}\norme{\rho}.}
Noting that $k\norm{Q_h\rho}_{m-1,h}\ls kh^{1-m}\norm{Q_h\rho}_{0,h}\ls kh^{2-m}\norme{\rho}\ls h^{1-m}\norme{\rho}$, we have
\eq{\norm{\ta_h}_{m,h}\ls k\norm{\ta_h}_{m-1,h}+h^{1-m}\norme{\rho},\quad 1\le m\le p.\label{estep2a}}
From a recursive use of the above estimate we have for $0\le m\le p$,
\eq{\label{estep2b}
\norm{\ta_h}_{m,h}&\ls k^{m}\norm{\ta_h}_0+\sum_{n=0}^{m-1}k^nh^{1-m+n}\norme{\rho} \ls k^m\norm{\ta_h}_0+h^{1-m}\norme{\rho}.}

\emph{Step 3.} In this step, we bound the $L^2$-norm of $\ta_h$ by its $(p-1)$-th order discrete norm.  Consider the following dual problem:
\begin{align}
-\triangle w-k^2w&=\ta_h \ \ \ \ \rm{in} \ \ \Om, \label{auxii1}\\
\frac{\partial w}{\partial n}-\i kw&=0 \ \ \ \ \ \rm{on} \ \ \Ga.
\end{align}
 Testing the conjugated \eqref{auxii1} by $e_h=\rho+\ta_h$, using the Galerkin orthogonality with $v_h=P_hw$, and using \eqref{ePh},  we get
\eqn{&(\rho+\ta_h,\ta_h)=a(e_h,w)-k^2(e_h,w)+\i k\langle e_h,w\rangle\db\\
&=a(e_h,w-P_hw)+\i k\pd{e_h,w-P_hw}-k^2(e_h,w-P_hw)\db\\
&=a(e_h,w-P_hw)+(e_h,w-P_hw)+\i k\pd{e_h,w-P_hw}-(k^2+1)(e_h,w-P_hw)\db\\
&=a(\rho, w-P_hw)+(\rho,w-P_hw)-(k^2+1)(\rho+\ta_h,w-P_hw)+\i k\pd{\rho+\ta_h,w-P_hw}.
}
And as a consequence,
\begin{align}\label{eduall2}
\norm{\ta_h}_0^2&=a(\rho, w-P_hw)-k^2(\rho,w-P_hw)+\i k\pd{\rho,w-P_hw}\\
&\quad-(k^2+1)(\ta_h,w-P_hw)+\i k\pd{\ta_h,w-P_hw}-(\rho,\ta_h)\nn\\
&\le \norme{\rho}\norme{w-P_hw}+(k^2+1)\abs{(\ta_h,w-P_hw)}\nn\\
&\quad+k\abs{\pd{\ta_h,w-P_hw}}+\norm{\rho}_0\norm{\ta_h}_0,\nn
\end{align}
where we have used \eqref{eac} to derive the last inequality.
Next we estimate the terms on the right hand side. 
Similar to Lemma~\ref{error2} we may show that
\begin{align}
 \norme{w-P_hw}& \ls \big(h+(kh)^p\big)\norm{\ta_h}_0,\qquad\|w-P_hw\|_0\ls h\big(h+(kh)^p\big)\norm{\ta_h}_0.\label{ewh-err2}
\end{align}
From Lemmas~\ref{lAhnorm}, \ref{lPh}--\ref{lQh},  and \eqref{ewh-err2},
\begin{align}
\abs{(\ta_h,w-P_hw)}&=\abs{(\ta_h,Q_hw-P_hw)}\label{ewh-err3}\\
&\le \norm{\ta_h}_{p-1,h}\norm{Q_hw-w+w-P_hw}_{1-p,h}\nn\\
&\ls \norm{\ta_h}_{p-1,h} h^p \big(h+(kh)^p\big)\norm{\ta_h}_0.\nn
\end{align}
On the other hand, from \eqref{etaGa} and \eqref{ewh-err2},
\eq{ \label{ewh-err4}\abs{\pd{\ta_h,w-P_hw}}&\ls 
\norml{\ta_h}{\Ga}\norml{w-P_hw}{\Ga}\\
&\ls\big(kh^{p-1/2}\norm{\ta_h}_{p-1,h}+h^{1/2}\norme{\rho}\big)h^{1/2}\big(h+(kh)^p\big)\norm{\ta_h}_0\nn\\
&\ls\big(kh^{p}\norm{\ta_h}_{p-1,h}+h\norme{\rho}\big)\big(h+(kh)^p\big)\norm{\ta_h}_0.\nn
}
 Finally, by plugging \eqref{erho} and \eqref{ewh-err2}--\eqref{ewh-err4} into \eqref{eduall2}, we have
\eq{\norm{\ta_h}_0\ls \big(h+(kh)^p\big)\norme{\rho}+ h^p \big(h+(kh)^p\big)k^2 \norm{\ta_h}_{p-1,h}.\label{estep1}
}

\emph{Step 4.} By combining \eqref{estep1} and \eqref{estep2b} with $m=p-1$, we have
\eqn{\norm{\ta_h}_0&\ls \big(h+(kh)^p\big)\norme{\rho}+h^p \big(h+(kh)^p\big)\big(k^{p+1}\norm{\ta_h}_0+k^2h^{2-p}\norme{\rho}\big)\\
&\ls \big(h+(kh)^p\big)\norme{\rho}+\big((kh)^{p+1}+k(kh)^{2p}\big)\norm{\ta_h}_0.}
Therefore, there exists a constant $C_0$ such that, if $k(kh)^{2p}\le C_0$, then
\eqn{\norm{\ta_h}_0\ls \big(h+(kh)^p\big)\norme{\rho}.}
Moreover \eqref{estep2a} (with $m=1$) and \eqref{enormh2} imply that
\eqn{\norm{\ta_h}_1\ls k\norm{\ta_h}_0+\norme{\rho}\ls \big(1+k(kh)^p\big)\norme{\rho}.}
Now the proof of the theorem follows from the above two estimates and \eqref{erho}. 
\end{proof}

From Theorem~\ref{thm-err-1} and Lemmas~\ref{lstau}--\ref{lapprox1}, we have the following corollary which gives preasymptotic estimates for $H^{p+1}$-regular solutions.
\begin{corollary}\label{cor-1}
Suppose  $C_{p-1,f,g}\ls 1$.  Then there exist constants $C_0, C_1, C_2$ independent of $k$ and $h$, such that if $k(kh)^{2p}\le C_0$ then the following estimates hold:
\begin{align}\label{ecor-1-a}
\norm{u-u_h}_1&\le C_1(kh)^p+C_2k(kh)^{2p},\\
k\|u-u_h\|_0&\le C_1(kh)^{p+1}+C_2k(kh)^{2p}.\label{ecor-1-b}
\end{align}
\end{corollary}

\rem 5.1.\ 
{\rm (a)} Preasymptotic error analysis and dispersion analysis are two main tools to understand numerical behaviors in short wave computations. The latter one which is usually performed on \emph{structured} meshes estimates the error between the wave number $k$ of the continuous problem and some discrete wave number $\om$ \cite{Ainsworth04, dbb99, harari97, ib95a, ib97, thompson94, thompson06}. In particular, it is shown for the FEM (cf. \cite{Ainsworth04, ib97}) that
\begin{align*}
k-\om=
O\big(k^{2p+1}h^{2p}\big) &\text{ if } k h\ll 1, 
\end{align*}
By contrast, our preasymptotic analysis gives the error between the exact solution $u$ and the discrete solution $u_h$ and works for \emph{unstructured} meshes.  Clearly, our pollution error bounds in $H^1$-norm coincide with the phase difference $\abs{k-\om}$ as above.

{\rm (b)} For problems with large wave number, a discrete solution of reasonable accuracy requires the pollution error $C_2k(kh)^{2p}$ to be small enough. From this point of view, our mesh condition $k(kh)^{2p}\le C_0$ is quite practical.

{\rm (c)} For the preasymptotic error estimates for the FEM in one dimension, we refer to \cite{ib95a,ib97}. For the case of higher dimensions, \cite{w, zw} give estimates under the mesh condition that $k(kh)^{p+1}\le C_0$. Our condition \eqref{econd1} gives larger range of $h$ than previous results in the case of $p>1$.

{\rm (d)} Error estimates in high order discrete Sobolev norms and in negative norms can also be derived. The details are omitted.

By combining Lemmas~\ref{depcomposition}, \ref{error2} and Theorem~\ref{thm-err-1} we have the following stability estimates for the FEM.
\begin{corollary}\label{cor-3}
Suppose the solution $u\in H^2(\Om)$. Under the conditions of Theorem~\ref{thm-err-1}, there holds the following estimate:
\begin{align*}
\norm{\na u_h}_0+k\norm{u_h}_0&\lesssim C_{f,g},
\end{align*}
and hence the FEM is well-posed.
\end{corollary}
\begin{proof}
It follows from Lemma~\ref{depcomposition}, Theorem~\ref{thm-err-1}, and Lemma~\ref{error2} that
\begin{align*}
\norme{u_h}\ls&\norme{u}+\norme{u-u_h}\\
&\ls \Big(1+\big(1+k(kh)^p\big)\big(h+(kh)^p\big) \Big)C_{f,g}\ls C_{f,g}.
\end{align*}
The proof is completed.
\end{proof}

\rem 5.2. {\rm (a)} This stability bound of finite element solution is of the same order as that of the continuous solution (cf. Lemma~\ref{depcomposition}).

{\rm (b)} When $k(kh)^{2p}$ is large, the well-posedness of the FEM in higher dimensions is still open.

\section{Preasymptotic error analysis of CIP-FEM}\label{sec-cipfem} In this section we prove preasymptotic error estimates of the CIP-FEM. The proofs of most results in this section are quite similar to the counterparts for the FEM, and will be either omitted or sketched by indicating the necessary modifications.  
We assume that the penalty parameters $\ga_0\le \ga_{j,e}\ls 1, \forall e\in\E_h^I$, $1\le j\le p$, where the constant $\ga_0$ will be specified later in Lemma~\ref{lga0}.

\subsection{Approximation properties} Similarly to Lemma~\ref{lapprox1}, we have the following lemma.
\begin{lemma}\label{lapprox1g}
Let $1\le s\le p+1$. Suppose $u \in H^s(\Omega)$. Then there exists $\hat{u}_h \in V_h$ such that
\begin{align}\label{elapprox1ag}
\norm{u-\hat{u}_h}_0+h E_\ga(u,\hat{u}_h)&\lesssim  h^s|u|_s.
\end{align}
\end{lemma}
\begin{proof}
Let $\hat{u}_h\in V_h$ be chosen as in Lemma~\ref{lapprox1}. Then, for $0\le j\le s$,
\begin{align}\label{inte}
\norm{u-\hat{u}_h}_{H^j(\T_h)}:=\left(\sum_{K\in\T_h}\norm{u-\hat{u}_h}_{H^j(K)}^2\right)^{\frac12} \lesssim h^{s-j}\abs{u}_{H^s(\Om)}.
\end{align}
In particular,
\begin{align}\label{elapprox1ag1}
\norm{u-\hat{u}_h}_0+h\norm{u-\hat{u}_h}_{1}&\lesssim  h^s|u|_s.
\end{align}
Next we estimate penalty terms in $E_\ga(u,\hat u_h)$ (cf. \eqref{e2.6a}). By an application of  the local trace inequality
\eq{\label{elti}
\norm{v}_{L^2(\p K)}^2\lesssim h^{-1} \norm{v}_{L^2(K)}^2+\norm{v}_{L^2(K)}\norm{\na v}_{L^2(K)},\quad \forall v\in H^1(K), K\in\M_h,}
the inverse inequality,  and  \eqref{inte}, we conclude that, for $j\leq s-1$,
\begin{align} \label{pe1}
&\sum_{e\in\E_h^{I}}\abs{\ga_{j,e}}\, h_e^{2j-1} \norm{\left[\frac{\pa^j \hat{u}_h}{\pa n_e^j}\right]}_{L^2(e)}^2 =\sum_{e\in\E_h^{I}}\abs{\ga_{j,e}}\, h_e^{2j-1} \norm{\left[\frac{\pa^j (u-\hat{u}_h)}{\pa n_e^j}\right]}_{L^2(e)}^2\\
&\lesssim  h^{2j-1}  \sum_{K\in \T_h}\sum_{e\subset\p K} \norm{\frac{\pa^j (u-\hat{u}_h)}{\pa n_e^j}}_{L^2(e)}^2\nn\\
&\lesssim  h^{2j-1}  \big( h^{-1}\norm{u-\hat{u}_h}_{H^j(\T_h)}^2+\norm{u-\hat{u}_h}_{H^j(\T_h)}\norm{u-\hat{u}_h}_{H^{j+1}(\T_h)}\big)\nn\\
&\lesssim  h^{2j-1} \cdot h^{2(s-j)-1} \abs{u}_{H^s(\Om)}^2= h^{2(s-1)} \abs{u}_{H^s(\Om)}^2.\nn
\end{align}
On the other hand, for $s\leq j\leq p$,
\begin{align}\label{pe2}
&\sum_{e\in\E_h^{I}}\abs{\ga_{j,e}}\, h_e^{2j-1} \norm{\left[\frac{\pa^j \hat{u}_h}{\pa n_e^j}\right]}_{L^2(e)}^2 \lesssim h^{2j-1} \sum_{K\in \T_h}\sum_{e\in\p K} \norm{\frac{\pa^j \hat{u}_h}{\pa n_e^j}}_{L^2(e)}^2\\
&\lesssim h^{2j-2}\abs{\hat{u}_h}_{H^j(\T_h)}^2 \lesssim h^{2(s-1)}\abs{\hat{u}_h}_{H^s(\T_h)}^2 \lesssim h^{2(s-1)}\abs{u}_{H^s(\Om)}^2.\nn
\end{align}
Then \eqref{elapprox1ag} follows by combining  \eqref{e2.6a}, \eqref{elapprox1ag1} and \eqref{pe1}--\eqref{pe2}. This completes the proof of the
lemma.
\end{proof}

If $u$ is the exact solution satisfying the decomposition $u=\ue+\ua$ as in Lemma~\ref{depcomposition}, then we may approximate $u$ by $\hat u_h=\widehat\ue_h+\widehat\ua_h$ and show the following estimate.
 \begin{lemma}\label{error2g}
Let $u$ be the solution to the problem \eqref{eq1.1a}-\eqref{eq1.1b}. Suppose $f\in L^2(\Om)$ and $g\in H^{1/2}(\Ga)$. Then
there exists $\hat{u}_h \in V_h$ such that
\begin{align}
\|u-\hat{u}_h\|_0+h\mathbb{E}_\ga(u,\hat{u}_h)&\lesssim \big(h^2+h(kh)^p\big) C_{f,g}, \label{u_err0g}
\end{align}
where $C_{f,g}$ is defined in Lemmas~\ref{depcomposition}.
\end{lemma}

\subsection{Discrete elliptic operator and discrete Sobolev norms} The following lemma determines the constant $\ga_0$.
\begin{lemma}\label{lga0} There exists a constant $\ga_0<0$, such that, if $\ga_0\le\ga_{j,e}\ls 1, $ for $ 1\le j\le p, e\in \E_h^I$, then
\eq{a_\ga(v_h,v_h)^{1/2}\eqsim\norm{\na v_h}_0, \quad\forall v_h\in V_h.\label{elga0}}
\end{lemma}
\begin{proof}
For any $e\in\E_h^I$, let $\Om_e$ be the union of two elements in $\T_h$ that share the common edge/face $e$.
From \eqref{elti} and the inverse inequality, we have
\eqn{h_e^{2j-1} \norm{\jm{\frac{\pa^j v_h}{\pa n_e^j}}}_{L^2(e)}^2
&\ls h_e^{2j-1}h_e^{-1}\sum_{K\subset\Om_e}\abs{v_h}_{H^{j}(K)}^2\ls \norm{\na v_h}_{L^2(\Om_e)}^2.}
Therefore, there exists a constant $\tilde{C}>0$ such that
\eqn{J(v_h,v_h)&=\sum_{j=1}^p\sum_{e\in\E_h^{I}}\ga_{j,e}\, h_e^{2j-1} \norm{\jm{\frac{\pa^j v_h}{\pa n_e^j}}}_{L^2(e)}^2\\
&\ge\sum_{j=1}^p\sum_{e\in\E_h^{I}}\ga_0\, h_e^{2j-1} \norm{\jm{\frac{\pa^j v_h}{\pa n_e^j}}}_{L^2(e)}^2\ge \ga_0\tilde{C}\norm{\na v_h}_0^2.}
Moreover,
\eqn{J(v_h,v_h)\ls\norm{\na v_h}_0^2.}
The above two estimates and  \eqref{eah} imply that
\eqn{(1+\ga_0\tilde{C})\norm{\na v_h}_0^2\le a_\ga(v_h,v_h)=\norm{\na v_h}_0^2+J(v_h,v_h)\ls\norm{\na v_h}_0^2.  }
Therefore, \eqref{elga0} holds if $-1/\tilde{C}<\ga_0<0$. This completes the proof of the lemma. \end{proof}

\rem 6.1. It follows from the proof of the above lemma and \eqref{e2.5}--\eqref{e2.5b} that, if $|\ga_{j,e}|\ls 1,  $ for $ 1\le j\le p, e\in \E_h^I$, then
\eqn{\norm{v_h}_{1,\ga}\eqsim\norm{v_h}_1, \quad\forall v_h\in V_h.}

\emph{In the rest of this section we assume that $\ga_0$ is determined by Lemma~\ref{lga0} and that  $\ga_0\le\ga_{j,e}\ls 1, $ for $ 1\le j\le p, e\in \E_h^I$.}

Define the elliptic projection $P_{h,\ga}$ as follows.
\eq{a_\ga(P_{h,\ga}\vp, v_h)+(P_{h,\ga}\vp,v_h)=a(\vp, v_h)+(\vp,v_h), \quad\forall v_h\in V_h, \vp\in V,\label{ePhg}}
where $a_\ga$ is defined in \eqref{eah}.
Then we have the following error estimates in $H^1$, $L^2$, and negative norms.
\begin{lemma}\label{lPhg} For any $-1\le j\le p-1$ and $\vp\in H^1(\Om)$,
\eqn{\norm{\vp-P_{h,\ga}\vp}_{-j}\ls h^{j+1}\inf_{\vp_h\in V_h}E_\ga(\vp,\vp_h).}
\end{lemma}
\begin{proof} \eqref{ePhg} can be rewritten as:
 \eq{a(\vp-P_{h,\ga}\vp, v_h)+(\vp-P_{h,\ga}\vp,v_h)=J(P_{h,\ga}\vp,v_h), \quad\forall v_h\in V_h, \vp\in V.\label{ePhga}}
From Remark 6.1, Lemma~\ref{lga0}, \eqref{ePhga}, \eqref{e2.5b}, and \eqref{e2.6a}, we have, for any $\vp_h\in V_h$,
\eqn{\|\vp_h-&P_{h,\ga}\vp\|_{1,\ga}^2\ls\norm{\vp_h-P_{h,\ga}\vp}_{1}^2\\
&\ls a_\ga(\vp_h-P_{h,\ga}\vp,\vp_h-P_{h,\ga}\vp)+(\vp_h-P_{h,\ga}\vp,\vp_h-P_{h,\ga}\vp)\\
&=a(\vp_h-\vp,\vp_h-P_{h,\ga}\vp)+(\vp_h-\vp,\vp_h-P_{h,\ga}\vp)+J(\vp_h,\vp_h-P_{h,\ga}\vp)\\
&\ls E_\ga(\vp,\vp_h)\norm{\vp_h-P_{h,\ga}\vp}_{1,\ga},}
Therefore, from the triangle inequality, we have
\eqn{E_\ga(\vp,P_{h,\ga}\vp)\le E_\ga(\vp,\vp_h)+\|\vp_h-P_{h,\ga}\vp\|_{1,\ga}\ls E_\ga(\vp,\vp_h)}
and hence
\eq{\label{elPhga}E_\ga(\vp,P_{h,\ga}\vp)\ls \inf_{\vp_h\in V_h}E_\ga(\vp,\vp_h).}
which implies that the lemma holds with $j=-1$.

Next we prove the  error estimates in $L^2$ (j=0) and negative norms ($1\le j\le p-1$)  by the duality argument (cf. \cite{bs08}).
For any $v\in H^j(\Om)$, let $w$ be the solution of the following problem:
\begin{align}\label{dp}
-\De w + w &=v\qquad in\qquad \Om,\\
\frac{\pa w}{\pa n}&=0\qquad on\qquad \pa \Om.\nn
\end{align}
Testing  the conjugated \eqref{dp} by $\vp-P_{h,\ga}\vp$ and using \eqref{ePhga}, \eqref{elPhga}, and Lemma~\ref{lapprox1g}, we get
\begin{align}
&(\vp-P_{h,\ga}\vp,v)=a(\vp-P_{h,\ga}\vp, w)+(\vp-P_{h,\ga}\vp,w)\\
&=a(\vp-P_{h,\ga}\vp, w-P_{h,\ga}w)+(\vp-P_{h,\ga}\vp,w-P_{h,\ga}w)+J(P_{h,\ga}\vp,P_{h,\ga}w)\nn\\
&\leq E_\ga(\vp,P_{h,\ga}\vp)E_\ga(w,P_{h,\ga}w)\ls \inf_{\vp_h\in V_h}E_\ga(\vp,\vp_h)\inf_{w_h\in V_h}E_\ga(w,w_h)\nn\\
&\ls \inf_{\vp_h\in V_h}E_\ga(\vp,\vp_h) h^{j+1}\norm{w}_{H^{j+2}(\Om)}\ls \inf_{\vp_h\in V_h}E_\ga(\vp,\vp_h)h^{j+1}\norm{v}_{H^j(\Om)},\nn
\end{align}
which implies that the lemma holds with $0\le j\leq p-1$.
This completes the proof of the lemma.
\end{proof}

Define $A_{h,\ga}: V_h\mapsto V_h$ by
\eq{(A_{h,\ga} v_h,w_h)=a_\ga(v_h,w_h)+(v_h,w_h),\quad\forall v_h,w_h\in V_h.\label{eAhg}}
Clearly, under the conditions of Lemma~\ref{lga0}, $A_{h,\ga}$ is symmetric and positive definite. Therefore we may define the powers of the operator $A_{h,\ga}$ by using its eigenvalues and eigenfunctions just like \eqref{eAhj}. And similarly to \eqref{enormh}, we 
introduce the following discrete $H^j$ norms on $V_h$ for any integer $j$:
\eq{\norm{v_h}_{j,h,\ga}=\norm{A_{h,\ga}^{j/2}v_h}_0.\label{enormhg}}
From the above definition, Lemma~\ref{lga0}, and Remark 6.1, it is clear that
\eq{\norm{v_h}_{0,h,\ga}=\norm{v_h}_0,\,\norm{v_h}_{1,h,\ga}=(A_{h,\ga}v_h,v_h)^{1/2}\eqsim\norm{v_h}_1\eqsim\norm{v_h}_{1,\ga}, \;\forall v_h\in V_h.\label{enormh2g}}

The following lemma parallel to Lemma~\ref{lAhinv} gives some inverse estimates for discrete functions. The proof is omitted.
\begin{lemma}\label{lAhinvg} For any integer $j$,
\eqn{\norm{v_h}_{j,h,\ga}\ls h^{-1}\norm{v_h}_{j-1,h,\ga},\quad\forall v_h\in V_h.}
\end{lemma}

Similarly to Lemma~\ref{lAhnorm}, we have the following lemma which gives a relationship between the discrete $H^{-j}$ norm ($j\ge 0$) and the $H^{-j}$ norm of discrete functions. Since its proof is almost the same as  that of Lemma~\ref{lAhnorm}, we omit it to save space. 
\begin{lemma}\label{lAhnormg} For any integer $0\le j\le p+1$, we have
\eq{\norm{v_h}_{-j,h,\ga}&\ls \sum_{m=0}^{j}h^{j-m}\norm{v_h}_{-m}, \quad\forall v_h\in V_h.\label{enormh-jag}
}
\end{lemma}

\subsection{Preasymptotic error analysis}\label{sec-preasy-dual-g}
The following Theorem gives preasymptotic error estimates for the CIP-FEM. The proof is omitted since it is quite similar to that of Theorem~\ref{thm-err-1} except the norm  $\norm{\cdot}_{j,h}$ should be replaced by  $\norm{\cdot}_{j,h,\ga}$ and the errors in the norm $\norme{\cdot}$ should be replaced by the errors measured by the function $\mathbb{E}$ defined in \eqref{e2.6b}.
\begin{theorem}\label{thm-err-1g}
Let $u$ and $u_h$ be the solutions to \eqref{eq1.1a}-\eqref{eq1.1b} and \eqref{ecipfem}, respectively. Then there exist a constant $C_0$ independent of $k$ and $h$, such that if
\begin{equation}\label{econd1g}
k(kh)^{2p}\le C_0,
\end{equation}
then the following error estimates hold:
\begin{align}\label{error-eh1g}
\mathbb{E}_\ga(u,u_h)&\lesssim \big(1+k(kh)^p\big)\inf_{z_h \in V_h}\mathbb{E}_\ga(u,z_h),\\
\|u-u_h\|_0&\lesssim \big(h+(kh)^p\big)\inf_{z_h \in V_h}\mathbb{E}_\ga(u,z_h).\label{error-eh2g}
\end{align}
\end{theorem}

From Theorem~\ref{thm-err-1g} and Lemmas~\ref{lstau} and \ref{lapprox1g}, we have the following corollary which gives preasymptotic estimates for $H^{p+1}$ regular solutions.
\begin{corollary}\label{cor-1g}
Let $u$ and $u_h$ be the solutions to \eqref{eq1.1a}-\eqref{eq1.1b} and \eqref{ecipfem}, respectively. Suppose  $C_{p-1,f,g}\ls 1$.  Then there exist constants $C_0, C_1, C_2$ independent of $k$ and $h$, such that if $k(kh)^{2p}\le C_0$ then the following estimates hold:
\begin{align}\label{ecor-1g-a}
\norm{u-u_h}_1&\le C_1(kh)^p+C_2k(kh)^{2p},\\
k\|u-u_h\|_0&\le C_1(kh)^{p+1}+C_2k(kh)^{2p}.\label{ecor-1g-b}
\end{align}
\end{corollary}

\rem 6.1.\
{\rm (a)} We have proven that the new CIP-FEM with real penalty parameters satisfies the same preasymptotic error estimates as those of FEM (cf. Theorem~\ref{thm-err-1} and Corollary~\ref{cor-1}). The mesh condition $k(kh)^{2p}\le C_0$ improves the previous results in \cite{w,zw} which require that $k(kh)^{p+1}\le C_0$ (for fixed $p>1$.) For preasymptotic analysis of the CIP-FEM with complex penalty parameters, we refer to \cite{w,zw}.

{\rm (b)} In the next section, we will show, via dispersion analysis and numerical examples, that the pollution error of the CIP-FEM may be reduced greatly by tuning the penalty parameters.

By combining Lemmas~\ref{depcomposition}, \ref{error2g} and Theorem~\ref{thm-err-1g} we have the following stability estimates for the CIP-FEM. The proof is similar to that of Corollary~\ref{cor-3} and is omitted.
\begin{corollary}\label{cor-3g}
Suppose the solution $u\in H^2(\Om)$. Under the conditions of Theorem~\ref{thm-err-1g}, there holds the following estimate:
\begin{align*}
\norm{\na u_h}_0+k\norm{u_h}_0&\lesssim C_{f,g},
\end{align*}
and hence the CIP-FEM is well-posed.
\end{corollary}

\section{Numerical examples}\label{sec-num}
In this section, we simulate the following two dimensional Helmholtz problem by FEM and CIP-FEM with $p=1, 2, 3$ on Cartesian meshes.
\begin{align}
-\De u - k^2 u &=f:=\frac{\sin(kr)}{r}  \qquad\mbox{in  } \Om, \label{num-eq-1}\\
\frac{\pa u}{\pa n} +\i k u& =g \quad\qquad\qquad\qquad\mbox{on } \Ga. \label{num-eq-2}
\end{align}
Here the computational domain $\Om$ is the unit square $(0,1)\times(0,1)$ and $g$ is so chosen that the exact solution is
\eq{u=\frac{\cos(k r)}{r}-\frac{\cos k+\i\sin k}{k\big(J_0(k)+\i J_1(k)\big)}J_0(k r)}
in polar coordinates, where $J_{\nu}(z)$ are Bessel functions of the first kind. We remark that this problem has been computed in \cite{fw09,w} by the linear FEM, CIP-FEM, and IPDG method on triangular meshes.


For any positive integer m, let $\T_{1/m}$ be the Cartesian grid that consists of $m^2$ congruent small squares of size $h = 1/m$. 
We remark that the number of total DOFs of both the FEM and CIP-FEM on $\T_{1/m}$ is $(pm)^2$.

Denote by $t:=kh$. For the CIP-FEM with $p=1, 2, 3$, we use the following penalty parameters which are obtained by a dispersion analysis for one dimensional problems such that the phase errors are entirely eliminated.

\noindent For p=1, let
\eq{\ga_{1,e}\equiv \frac{t^2(\cos t+2)+6\cos t-6}{12(1-\cos t)^2}; \label{opt-ga-p1}}
For p=2, let
\begin{align}\label{opt-ga-p2}
\ga_{1,e}\equiv& \frac{t^2\big(2\cos\frac{t}{2}+1\big)+12\cos^2\frac{t}{2}-12}{768\big(\sin^6\frac{t}{4}-\sin^8\frac{t}{4}\big)},\\
\ga_{2,e}\equiv
& \frac{\nfrac{t^2\big(8\sin^6\frac{t}{4}+12\sin^4\frac{t}{4}-30\sin^2\frac{t}{4}+15\big)}{
\quad-160\sin^6\frac{t}{4}+400\sin^4\frac{t}{4}-240\sin^2\frac{t}{4}}}{61440\big(\sin^{10}\frac{t}{4}-2\sin^8\frac{t}{4}+\sin^6\frac{t}{4}\big)};\nn
\end{align}
For p=3, let
\begin{align}\label{opt-ga-p3}
\ga_{1,e}\equiv& \frac{2t^2\big(36\cos\frac{t}{3} + 9\cos\frac{2t}{3} + 2\cos t + 13\big)+240(\cos t-1)}{480\big(2\cos\frac{t}{3} + 1\big)^2\big(4\cos\frac{t}{3} -1\big) \big(\cos\frac{t}{3} - 1\big)^3}\\
\ga_{2,e}\equiv&\frac{2t^2(\cos\frac{t}{3} + 28\cos\frac{2t}{3} + \cos\frac{4t}{3} - \cos t + 31)-120\sin^2\frac{t}{3}(2\cos\frac{t}{3} + 1)^2}{34560(2\cos\frac{t}{3} + 1)^3(\cos\frac{t}{3}-1)^4}\nn\\
\ga_{3,e}\equiv 
& \frac{\nfrac{36t^2\big(\cos(2t)+201\cos\frac{t}{3} + 93\cos\frac{2t}{3} + 24\cos\frac{4t}{3} - 3\cos\frac{5t}{3}+ 38\cos t }{\quad   +66\big)+504(\cos t-1)\big(36\cos\frac{t}{3}+9\cos\frac{2t}{3}+2\cos t+13\big)}}{6531840(\cos\frac{t}{3} - 1)^4(2\cos\frac{t}{3} + 1)^5}.\nn
\end{align}
  We remark that the penalty parameters in \eqref{opt-ga-p1} for $p=1$ was first given in \cite{zbw}. 
Although these parameters are derived for one dimensional problems, we use them in our computations for the two dimensional problem since we are using Cartesian grids.   
A detailed dispersion analysis for the CIP-FEM in both one and two dimensions will be reported in another work.

  From Theorem \ref{thm-err-1} (cf. Remark 4.1) and Theorem \ref{thm-err-1g}, the error of the FE  or CIP-FE solution $u_h$ in the $H^1$-seminorm is bounded by
\eq{\norm{\na (u-u_h)}_{L^2(\Om)}&\le C_1(kh)^p+C_2k(kh)^{2p},\label{num-err-1}}
for some constants $C_1$ and $C_2$ if $k(kh)^{2p}\leq C_0$. The second term on the right hand side of \eqref{num-err-1} is the so-called pollution error. We now present numerical results to verify the above error bounds.

\begin{figure}[htbp]
\begin{center}
\includegraphics[width=0.49\textwidth]{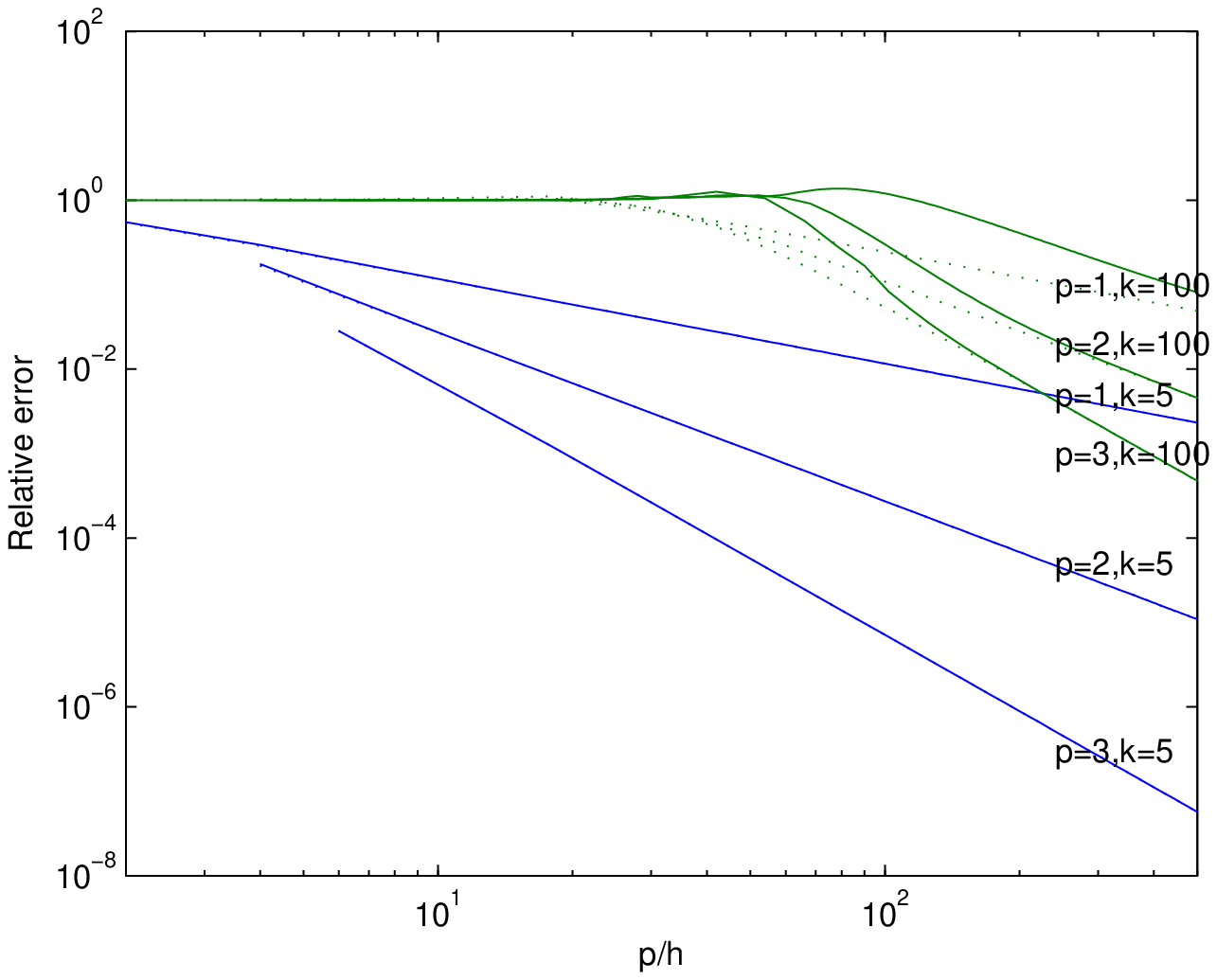}
\includegraphics[width=0.49\textwidth]{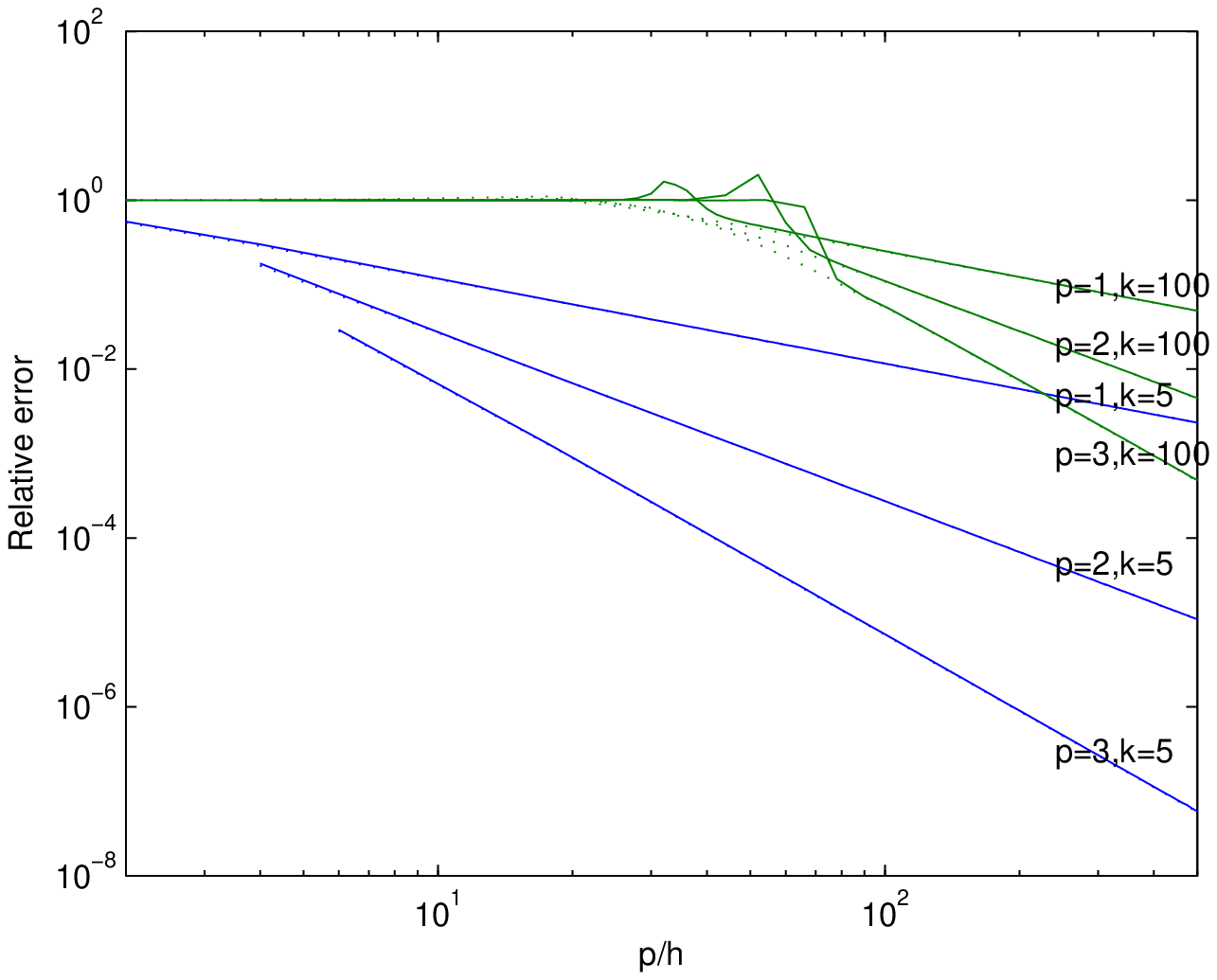}
\caption{Left graph: the relative error of the FE solution and the relative
error of the FE interpolation (dotted) in $H^1$-seminorm for $k = 5, 100$ and $p=1, 2, 3$, respectively. Right graph:
corresponding plots for CIP-FE solutions with parameters given by \eqref{opt-ga-p1}--\eqref{opt-ga-p3}.}
\label{fig:1}
\end{center}
\end{figure}

Figure~\ref{fig:1} plots the relative errors in $H^1$-seminorm of the FE solutions, the CIP-FE solutions with penalty parameters given by \eqref{opt-ga-p1}--\eqref{opt-ga-p3}, and the FE interpolations for $p=1,2,$ and $3$, respectively. It is shown that for $k=5$ the relative errors of both FE solutions and CIP-FE solutions fit those of the corresponding FE interpolations very well, which means the pollution errors do not come out for small $k$. For $k=100$, the relative errors of the FE solutions first stay around $100\%$, then decay slowly on a range starting with a point far from the decaying point of the corresponding FE interpolations, and then decays at a rate greater than $-p$ in the log-log scale but converges as fast as the FE interpolations (with slope $-p$) for small h. 
Such a behavior show clearly the effect of pollution of the FEM for large $k$ and $h$. The CIP-FE solutions behave similarly as the FE solutions but the pollution range of the former for each $p$ is much smaller than  that of the later, which means that the pollution effect is greatly reduced. To see this more intuitively we plot the relative errors of both methods for $p=1, 2, 3$ and $k=1, 2, \cdots, 1000$ with fixed $kh/p=1$ in one figure (see Figure~\ref{fig:2}). One can see that the pollution error of the FEM (for $p=1, 2,$ or $3$) becomes dominated when $k$ is greater than some value less than $50$, while the pollution error of the CIP-FEM (for $p=1, 2,$ or $3$) is almost invisible for $k$ up to $1000$. If we take a very close look at the relative error curve of the linear CIP-FEM ($p=1$), we may find that it increases very slowly, which means that the pollution effect is still there but very small.

\begin{figure}[htbp]
\begin{center}
\includegraphics[width=0.49\textwidth]{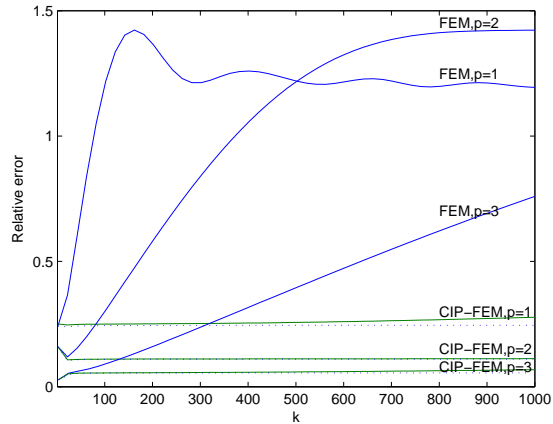}
\caption{The relative errors of the FE solutions, the CIP-FE solutions with parameters given by \eqref{opt-ga-p1}--\eqref{opt-ga-p3}, and the FE interpolations (dotted), with mesh size $h$ determined by $kh/p = 1$ for $p=1, 2, 3$, respectively.}
\label{fig:2}
\end{center}
\end{figure}

Next we verify more precisely the pollution term in \eqref{num-err-1}. To do so, we introduce the definition of the critical mesh size with respect to a given relative tolerance \cite{w}.

\begin{definition}
Given a relative tolerance $\ep$, a wave number $k$ and the polynomials' degree $p$, the critical mesh size $h(k,p,\ep)$ with respect to the relative tolerance $\ep$ is defined by the maximum mesh size such that the relative errors of the CIP-FE solution (or the FE solution) in $H^1$-seminorm is less than or equal to $\ep$.
\end{definition}

\begin{figure}[htbp]
\begin{center}
\includegraphics[width=0.49\textwidth]{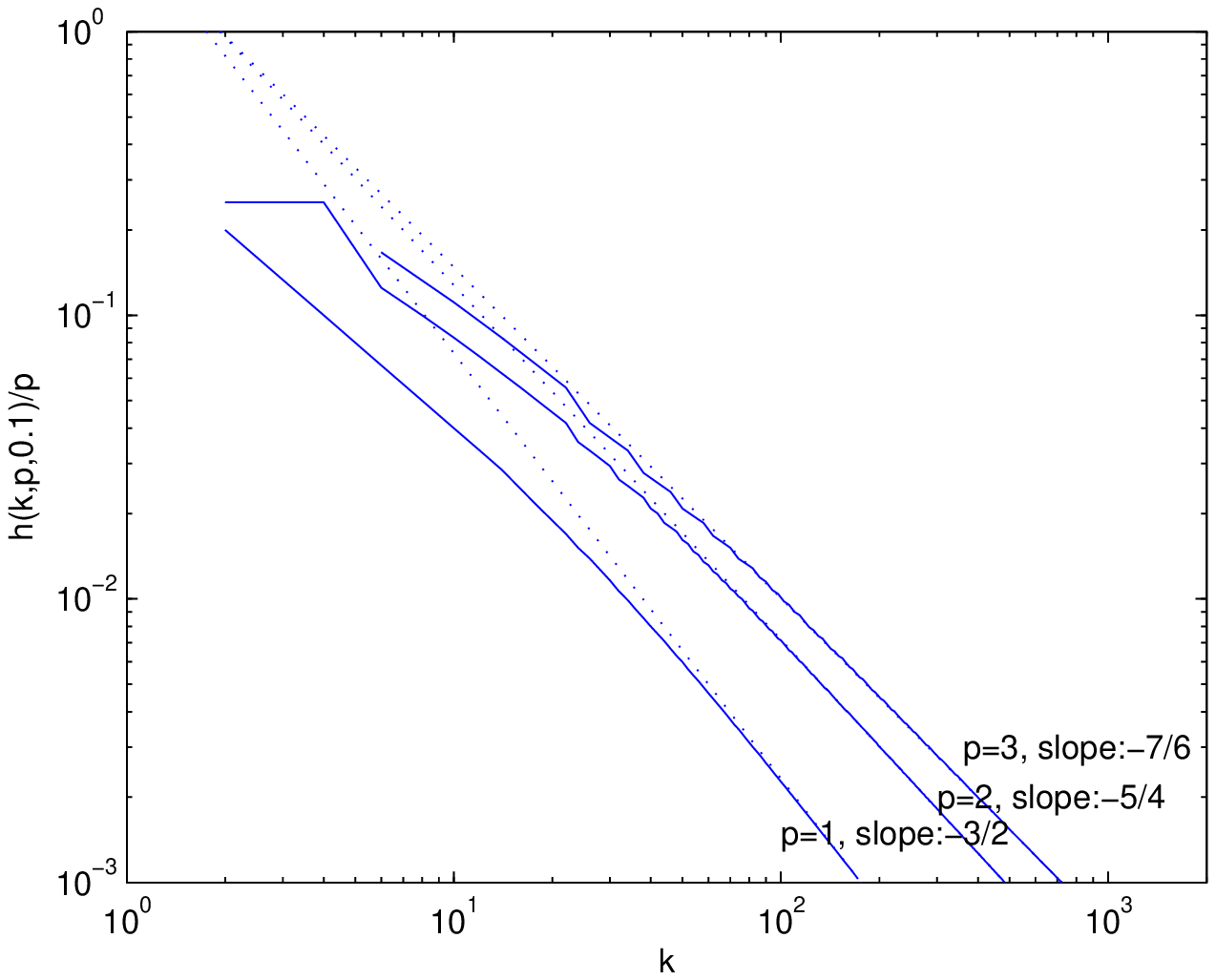}
\includegraphics[width=0.49\textwidth]{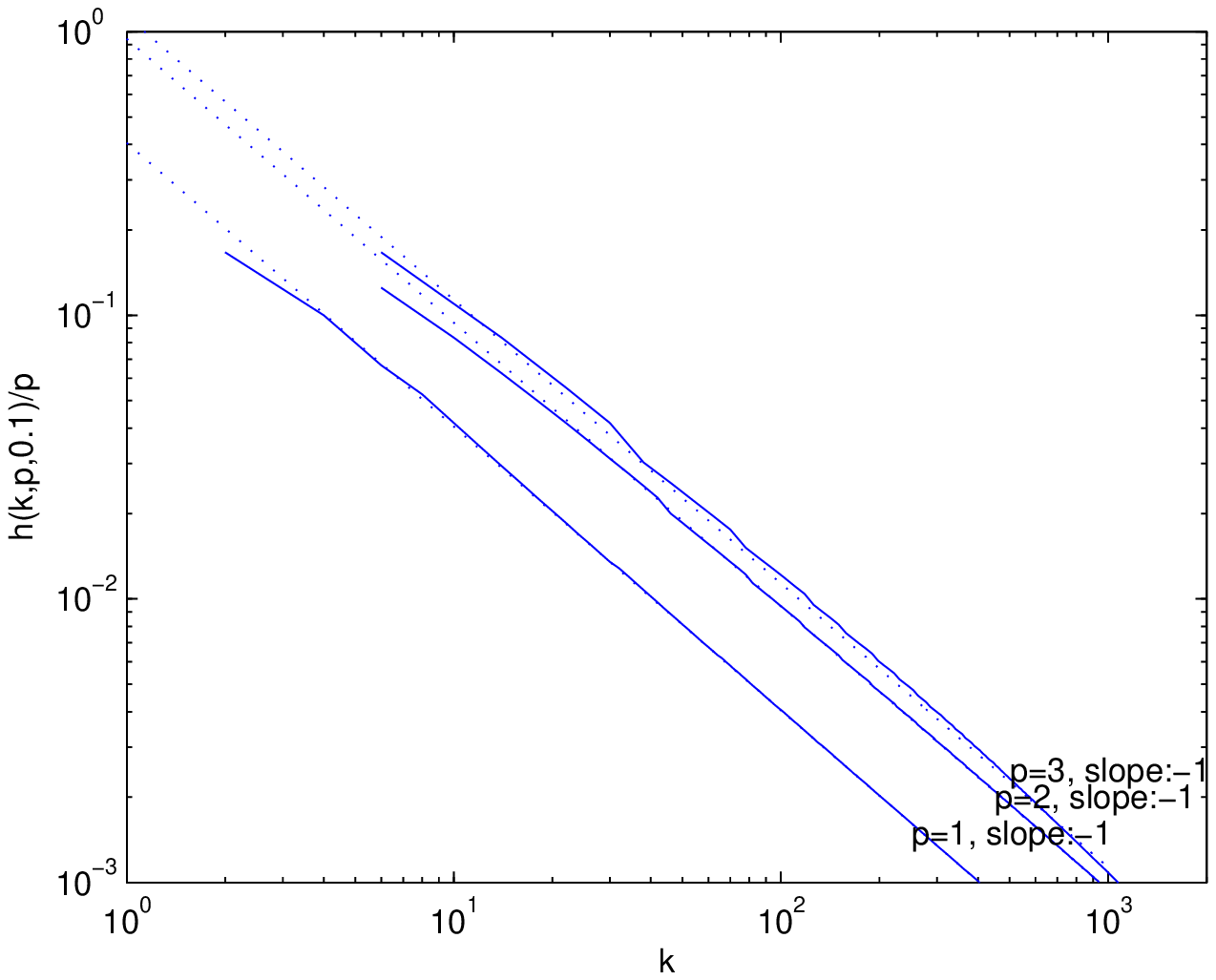}
\caption{The critical mesh size $h(k,p,0.1)$, $p=1,2,3$, versus k for the FEM (left) and the CIP-FEM (right) with parameters given by \eqref{opt-ga-p1}--\eqref{opt-ga-p3}, respectively. The dotted lines indicate reference slopes. }
\label{fig:3}
\end{center}
\end{figure}

It is clear that, if the pollution term of the FE solution in \eqref{num-err-1} is of order $k(kh)^{2p}$, then $h(k,p,\ep)$ should be proportional to $k^{-\frac{2p+1}{2p}}$ for $k$ large enough. This is verified by the left graphs of Fig \ref{fig:3} which plots $h(k,p,0.1),\ p=1,2,3$, the critical mesh size with respect to the relative tolerance $10\%$, versus $k$ for the FE solutions. The right graph of Fig \ref{fig:3} shows that $h(k,p,0.1)$ for the the CIP-FE solution is proportional to $k^{-1}$ which means the pollution effect does not show up yet in the computations.



\end{document}